\documentclass[10pt,a4wide]{article}

\usepackage{amssymb,amsmath,amsthm}	% Mathematische Symbole
\usepackage{geometry}
\usepackage{graphicx}      										% Grafiken einbinden
\graphicspath{{figs/}}
\usepackage{tikz}          										% LaTeX-Grafiken erzeugen
\usepackage{url}
\usepackage{dsfont}
\usepackage{amsmath}
\usepackage{nicefrac}
\usepackage{xspace}
\usepackage{framed}
\usepackage[left,inline]{showlabels}
\usepackage{hyperref}
\usepackage{cite}
\usepackage[textsize=footnotesize,backgroundcolor=yellow!70,bordercolor=orange]{todonotes}

\newcommand{\de}{\mathrm{d}}

\newtheorem{definition}{Definition}[]
\newtheorem{lemma}{Lemma}[]

\newtheorem{remark}{Remark}[]

\newtheorem{proposition}{Proposition}[]

\newtheorem{assumption}{Assumption}[]
\newcommand{\R}{\ensuremath{\mathbb{R}}}

\newcommand{\N}{\ensuremath{\mathbb{N}}}
\DeclareMathOperator*{\argmin}{argmin}

\begin{document}
	
	%-- TITEL ----------------------------------------------------------%
	\title{Continuous limits of  residual neural networks in case of large input data}
	\author{M.~Herty\thanks{Institut f\"{u}r Geometrie und Praktische Mathematik (IGPM) -- RWTH Aachen University -- Templergraben 55, 52062 Aachen (Germany) -- \texttt{herty@igpm.rwth-aachen.de}} ,\ A.~Th\"{u}nen\thanks{Institut f\"{u}r Mathematik -- TU Clausthal -- Erzstra\ss e 1,
	38678 Clausthal-Zellerfeld (Germany) -- \texttt{anna.thuenen@tu-clausthal.de}} ,\ T.~Trimborn\thanks{NRW.BANK -- Kavalleriestraße 22, 40213 D\"{u}sseldorf (Germany) -- \texttt{torsten.trimborn@nrwbank.de}} ,\ G.~Visconti\thanks{Department of Mathematics ``G.~Castelnuovo'' -- Sapienza University of Rome -- P.le Aldo Moro 5, 00185 Roma (Italy) -- \texttt{giuseppe.visconti@uniroma1.it}}}
	\maketitle

	%-- INHALTSVERZEICHNIS ----------------------------------------------------------%

	\begin{abstract}
		Residual deep neural networks (ResNets) are mathematically described as interacting particle  systems. In the case of  infinitely many layers the ResNet leads to a system of  coupled system of ordinary differential equations known as neural differential equations. For large scale input data we derive a mean--field limit and show  well--posedness of the resulting  description. Further, we analyze the existence of solutions to the training process by using both a controllability and an optimal control point of view. Numerical investigations based on the solution of a formal optimality system illustrate the theoretical findings.
	\end{abstract}
	
	
	
	\section{Introduction} \label{sec:introduction}
	
	In the last years, there has been a growing interest in machine learning and data science applications~\cite{Wooldridge2020,LalmuanawaHussainChhakchhuak2020,MuellerVincentBostrom2016}, e.g.~in the fields of recognition of human speech~\cite{MengistuRudzicz2011},  competition at the highest level in strategic game systems~\cite{LiXingFenghuaCheng2020},  intelligent routing in content delivery networks~\cite{FadlullahZubairMaoTangKato2019}, and  autonomously operating cars~\cite{STERN2018}. The intersection between mathematics and artificial intelligence has been mainly based on using machine learning tools to resolve bottlenecks in existing numerical methods, e.g.~to replace parameter optimization, parameter--identification and data assimilation methods, or shock--detection techniques for non--oscillatory reconstructions, or to model physics--based operators through experimental data and uncertainty quantification. We refer to  ~\cite{Mishra2019,LyeMishraRay2020,ZhangGuoKarniadakis2020,RaissiPerdikarisKarnidakis2019,DiscacciatiHesthavenRay2020,RayHesthaven2019,MagieraRayHesthavenRohde2020,RayHesthaven2018} for additional references on this topics. Here, we contribute to a  framework for a particular class of learning--based methods, the deep residual neural networks (ResNets), using a description based on partial differential equations, more precisely, linear kinetic equations. This formulation allows to apply different techniques to analyze theoretical properties of the underlying neural network. 
	
	First, we briefly recall the residual neural networks (ResNet), see also equation \eqref{eq:resnet}. 	
	 Given a set of input data or measurements $x_i^0, i=1,\dots,M$, the ResNet propagates those through discrete entities, the layers $\kappa=0,\dots,L+1$, to provide a state prediction $x_i(L+1)$. The dynamics depends on (a large set of) parameters, called weights $w(\kappa)$ and bias $b(\kappa)$. Their values are chosen in an optimization procedure called training to predict given reference data $y_i$. This training procedure is performed by  minimizing a given distance $\ell$ between  predictions of the network and the given reference values.
	
	Neural networks, and in general machine learning models, are typically processed on very large data sets, formally $M\to\infty$, making both the forward and the training processes computationally costly. An attempt to provide a more synthetic, and statistical, description of the neural network dynamics has been investigated in~\cite{HTV_MFResNet}, where a kinetic formulation of neural differential equations is proposed. However, neither the training nor the well--posedness have been analyzed so far. Using a mean--field or kinetic description of large--scale neural networks has so far only be discussed for particular examples in a few recent manuscripts~\cite{Crevat2019,MeiMontanariNguyen2018,SirignanoSpiliopoulos2020a,SirignanoSpiliopoulos2020b,BaccelliTaillefumier2019,TrimbornGersterVisconti2020}. A general investigation in particular in view of large input data is to the best of our knowledge still open.
	
	In this work we contribute towards the mean--field description of neural networks by  discussing well--posedness of the mean--field residual neural network using results obtained in the context of pedestrian dynamics~\cite{tosin}. The arising equation is similar to the 
	mean--field model formally proposed in~\cite{HTV_MFResNet}. However, therein the training process has not been discussed. Here, we follow 
	two directions. Training is formulated as a controllability problem for the mean--field equation. This problem allows for solutions for very particular initial and reference data. In the general case, the training process is considered as an optimal control problem with constraints given by the mean--field equation. Here, the derived continuous dependence on the parameters is used to show existence of optimal weights and bias. A numerical method for computing those based on the mean--field equation is implemented and computational results are presented.

	The structure of the paper is shortly summarized here. In Section~\ref{sec:preliminary} we introduce deep residual neural networks and discuss formally the equations resulting in time and in the mean--field continuous limits. Section~\ref{sec:analysis} is the main part of this work since we provide a rigorous analysis of the mean--field equation. In Section~\ref{sec:training} we discuss a computational technique for the training of the mean--field neural network and numerical experiments are performed. Finally, we conclude the paper in Section~\ref{sec:conclusion} proposing also research perspectives.
	
	%% MH: not necessary. \todo[inline]{Shall we clean the reference list? Maybe too bib item.\\Shall we pay more attention on citing Carrillo's works and in general editors' works?}
		
	\section{Continuous limits of residual neural networks} \label{sec:preliminary} 
	
	Let us consider a set of $M\in\N$, $M\gg 1$, input data characterized by $d$ measurements. In terms of neural networks, each measurement represents a feature of the given input.
	Without loss of generality it is possible to assume that the value of each feature is one--dimensional so that each input data can be described by $x_i^0\in\R^d$, $i=1,\dots,M$.
		
	As starting point we consider deep Residual Neural Networks (ResNets). Its structure is given by  $L$ hidden layers, with labels $\{1,\dots,L\}$, and in each layer the number of neurons is given by $N_\kappa$, $\forall \kappa=1,\dots,L$. We use the indices $\kappa=0$ and $\kappa=L+1$ to denote the input layer and the output layer, respectively. The state of the $i$--th input signal at the $\kappa$--th layer is $x_i(\kappa)\in\R^{N_\kappa}$ and $N_0=d$. The final state $x_i(L+1)\in\R^{N_{L+1}}$ is called {\em output } or prediction of the network.
	
	Each input signal $x_i^0$ propagates according to the deterministic dynamics~\cite{He2015DeepRL}:
	\begin{equation} \label{eq:resnet}
		\begin{cases}
		x_i(\kappa+1)=A(\kappa)x_i(\kappa)+\Delta t \, \sigma\big( w(\kappa) x_i(\kappa)+b(\kappa) \big), \quad \kappa=0,\dots,L\\
		x_i(0)=x_i^0.
		\end{cases}
	\end{equation}
	Here, $w(\kappa) \in\R^{N_{\kappa+1}\times N_\kappa}$ are the weights and $b(\kappa)\in\R^{N_{\kappa+1}}$ the bias and $\Delta t>0$ indicates a (pseudo) time step. The vectors $(w,b)$ define the  parameters of the network. The matrix $A(\kappa) \in R^{N_{\kappa+1}\times N_\kappa}$ is a deterministic matrix which reduces to an identity matrix under Assumption~\ref{ass:neurons}, cf.~the next section, and therefore we do not provide a rigorous definition. The function $\sigma: \R\to \R$ is called activation function of the neurons and it is applied component wise in~\eqref{eq:resnet}. Examples of activation functions include the identity function $\sigma(x) = x$, the rectified linear unit (ReLU) function $\sigma(x) = \max\{0,x\}$, the sigmoid function $\sigma(x) = \frac{1}{1+\exp(-x)}$, the hyperbolic tangent function $\sigma(x) = \tanh(x)$ and the growing cosine unit (GCU) function $\sigma(x)=x\cos(x)$. 
	
	The parameters $w$ and $b$ are  chosen in a {\em training process} in order to have ResNet solve a given learning problem.  In {\em  supervised training } we have that \emph{desired} outputs, the targets or reference values, are provided along with the input data. The network processes the input data and then compares the predictions against the targets $\{y_i\}_{i=1}^M$. The error is then propagated back through the network with the aim of optimizing the parameters. This process occurs many times on a set of data which is typically named as training data set and several approaches are known, as e.g.  stochastic gradient descent~\cite{haber2018look} or ensemble Kalman filter~\cite{Kovachki2019,Watanabe1990,Alper}, and they are related to the choice of  the {\em loss functions }. Typical loss functions~\cite{LossFunctions}, for instance such as Mean Squared Error, the Mean Absolute Error and the Categorical Cross--Entropy, can be all written as
	\begin{equation} \label{eq:loss}
		\frac{1}{M} \sum_{i=1}^M \ell(x_i(L+1)-y_i),
	\end{equation}
	for suitable choices of a (differentiable) function $\ell:\R^{N_{L+1}}\to\R^+_0$.
	
%	In this work the training is tackled in the mean--field limit of~\eqref{eq:resnet}, which is presented in the next subsection.
	
	\subsection{Neural differential equations and mean--field limit}
	
	In~\eqref{eq:resnet} the layers define a discrete structure within the ResNet. In order to compute the time continuous limit of~\eqref{eq:resnet} we interpret the layers as discrete times where the propagation of the input signal is evaluated. To this end we need to introduce the following assumption.
	
	\begin{assumption} \label{ass:neurons}
		The number of neurons in each layer is fixed and determined by the dimension of the input data. Namely, $N_\kappa = N = d$, $\forall\,\kappa=1,\dots,L+1$.
	\end{assumption}

	This assumption typically underlies the derivation of neural differential equations, see also~\cite{chen2018neural}. Note that even the choice $d=1$ is possible and, in addition, that networks of this type  have been already proved to satisfy different formulations of the universal approximation theorem~\cite{UniversalApproximator,LuLu2020,kidger2020universal}. Further, they have been also applied to several (real--world) training problems~\cite{GebhardtTrimborn,Bobzin2021}.
	
	Under Assumption~\ref{ass:neurons} and interpreting $\Delta t$ as the size of a time step, \eqref{eq:resnet} is an explicit Euler discretization of an underlying differential equation. Namely, in the limit $\Delta t\to 0^+$ and $L\to \infty$ such that $\Delta t(L+2) \to T$, \eqref{eq:resnet} formally leads to
	\begin{equation} \label{eq:micro}
	\begin{cases}
	\displaystyle{\frac{\mathrm{d}}{\mathrm{d}t}} x_i(t) = \sigma\big( w(t) x_i(t) + b(t) \big), \quad t\in[0,T]\\
	x_i(0) = x_i^0,
	\end{cases}
	\end{equation}
	for each $i=1,\dots,M$. The system of differential equations~\eqref{eq:micro} describes the time propagation of each measurement $x_i(t)\in\R^d$, starting from the initial condition $x_i^0\in\R^d$ fixed by the input data. It is known as {\em neural differential equation}. 
	\par 
	 The parameters of the network are given by the time dependent weights $w(t)\in\R^{d\times d}$ and by the time dependent bias $b(t)\in\R^d$, $\forall\,t\geq0$.  By the Picard--Lindel\"{o}f Theorem, existence and uniqueness of a solution to~\eqref{eq:micro} is guaranteed as long as the activation function $\sigma$ satisfies the Lipschitz condition and $t\mapsto w(t)$, $t\mapsto b(t)$ are continuous. Notice that the loss functional~\eqref{eq:loss} reads 
	\begin{equation} \label{eq:lossMicro}
		\frac{1}{M} \sum_{i=1}^M \ell(x_i(T)-y_i),
	\end{equation}
	 where $x_i(T)$ represents the state at time $T>0$ obtained with~\eqref{eq:micro}.
	
	It is clear from~\eqref{eq:micro} that the computational and memory cost of the neural network still increases with the dimension of the data set, i.e.~$M$. A way to overcome this problem is introducing a statistical interpretation of the neural network by computing the mean--field limit of the neural differential equations~\eqref{eq:micro} for $M\to \infty.$ In the limit of infinitely many data we formally obtain the linear equation 	
	\begin{equation} \label{eq:meso}
		\begin{cases}
			\partial_t f(t,x) + \nabla_x\cdot \Big(\sigma\big(w(t)x+b(t)\big)f(t,x)\Big) = 0, \quad t>0 \\
			f(0,x) = f_0(x),
		\end{cases}
	\end{equation}
	which describes the evolution of the distribution $f:\R_{\geq0}\times\R^d\to\R_{\geq0}^d$ of the data. The initial condition $f_0(x)$ is obtained as limit of the input data. Since~\eqref{eq:meso} preserves the mass, $f(t,x)$ is  a probability distribution $\forall\,t>0$ provided that $f_0$ is. We point out that in the mean--field limit any information on the network output of a precise measurement is lost. In fact, \eqref{eq:meso} provides only a statistical information on the neural network propagation and, thus, of the learning problem. The well--posedness of equation~\eqref{eq:meso} and the convergence of~\eqref{eq:micro} to~\eqref{eq:meso} as $M\to\infty$ is proven in $1$--Wasserstein distance, see Section~\ref{ssec:wellposed}.
	
	\section{Analysis of the mean--field limit} \label{sec:analysis}
	
	In this section, we discuss the mean--field limit of the system~\eqref{eq:micro} and related minimization problems following results of  \cite[Section 6 and 7]{tosin}. We start with recalling some preliminary notation and refer to \cite{ambrosioGradientFlowsMetric2008,VillaniBook} for more details.
	
	Let $\mathcal{P}(\R^d)$ the set of real--valued probability measures defined on $\R^d$ and, for $p \geq 1$, we denote
	by $\mathcal{P}_p(\R^d) \subset \mathcal{P}(\R^d)$ the set of probability measures with finite $p$--th moment, i.e.
	$$
		\mathcal{P}_p(\R^{d}) = \left\{ \mu\in\mathcal{P}(\R^d) \colon \int_{\R^{d}} |x|^p \de \mu(x) < +\infty \right\}.
	$$
	Throughout the paper we denote by $\mu_t$ a time dependent probability measure for $t\in\R^+_0$.
	
	Given a map $\gamma \colon \R^d \to \R^{d}$, the push--forward of $\mu\in\mathcal{P}(\R^d)$ through $\gamma$ is defined for every Borel set $A \subset \R^d$ as the unique probability measure $\gamma \# \mu$ such that $ \gamma \# \mu (A) \colon = \mu(\gamma^{-1}(A))$. Given two probability measures $\mu$, $\nu \in \mathcal{P}(\R^d)$, a probability measure $\pi$ on the product space $\mathbb{R}^d \times \mathbb{R}^d$ is said to be an admissible transport plan from $\mu$ to $\nu$ if the following properties hold: 
	\begin{equation*}
	\label{eq:transp_plan}
	\int_{y\in \mathbb{R}^d} \,\de \pi(x,y) = \de\mu(x), \quad \int_{x\in \mathbb{R}^d} \,\de\pi(x,y)= \de\nu(y).
	\end{equation*}
	We denote the set of admissible transport plans from $\mu$ to $\nu$ by $\Pi(\mu,\nu)$. Note that the set $\Pi(\mu, \nu)$ is always nonempty, since the product $\mu \nu \in \Pi(\mu, \nu)$. The cost of each admissible transport plan $\pi$ from $\mu$ to $\nu$ can be defined as follows: 
	$$J[\pi]\colon = \int_{\R^{2d}} |x-y|^p \de\pi(x,y),$$
	where $|\cdot|$ represents the Euclidean norm on $\R^d$. 
	A minimizer of $J$ in $\Pi(\mu, \nu)$ always exists. Thus for any two measures $\mu, \nu \in \mathcal{P}_p(\R^d)$, one can define the following metric 
	$$W_p(\mu, \nu) \colon = \left(\min\limits_{\pi\in \Pi(\mu, \nu)} J[\pi]\right)^{\frac{1}{p}},$$
	which is called the $p$--Wasserstein distance. The set of transport plans $\pi\in\Pi(\mu,\nu)$ achieving this optimal value is denoted by $\Pi_0(\pi,\nu)$ and is referred to as the
	set of optimal transport plans between $\mu$ and $\nu$. The space of probability measures $\mathcal{P}_p(\R^d)$ endowed with the $p$--Wasserstein distance is called the Wasserstein space of order $p$.
	
	\begin{table}[t]
			\caption{List of the microscopic and mean--field quantities.}
			\begin{center}
				\begin{tabular}{cc|cc}
					\hline
					\multicolumn{2}{c|}{Microscopic} & \multicolumn{2}{|c}{Mean--field}\\
					\hline\hline
					Trajectories of~\eqref{eq:micro} & $x_i(t)\in\R^d$ & Weak solution of~\eqref{eq:meso} & $f_t\in\mathcal{P}_1(\R^{d})$ \\
					Augmented trajectories of~\eqref{003} & $(x_i(t),\tau_i(t))\in\R^{d+1}$ & Weak solution of~\eqref{wmf-mh} & $F_t\in\mathcal{P}_1(\R^{d+1})$ \\
					Target data & $y_i\in\R^d$ & Target measure & $g\in\mathcal{P}_1(\R^{d})$ \\
					Loss function~\eqref{eq:lossMicro} & $\ell\colon\R^d\to\R$ & Loss function~\eqref{eq:lossMF} & $\tilde{\ell}\colon\R^d\to\R$ \\
					\hline
				\end{tabular}
			\end{center}
			\label{tab:mathSymbols}
	\end{table}
	
	Finally, in order to help the reader, we report in Table~\ref{tab:mathSymbols} a list of the microscopic and mean--field objects used in the analysis performed in the subsequent sections.
		
	\subsection{Well--posedness of weak solutions} \label{ssec:wellposed}
	
	We notice that the microscopic system~\eqref{eq:micro} describing a neural differential equation can be recast as an autonomous system using the auxiliary variables $\tau_i=\tau_i(t)\in\mathbb{R}$ for $i=1,\dots, M$:
	\begin{equation} \label{003}
	\begin{cases}
	\displaystyle{\frac{\mathrm{d}}{\mathrm{d}t}} x_i(t) = \sigma\big( w(\tau_i(t)) x_i(t) + b(\tau_i(t)) \big), & x_{i}(0)=x_i^0 \\[2ex]
	\displaystyle{\frac{\mathrm{d}}{\mathrm{d}t}} \tau_i(t) = 1, & \tau_i(0)=0.
	\end{cases}
	\end{equation}
	In the following, the right hand side of~\eqref{003} will be compactly denoted
	using the function
	\begin{equation} \label{001}
	\begin{aligned}
	G: \mathbb{R}^{d+1} &\to \mathbb{R}^{d+1} \\
	(x,\tau) &\mapsto \Big( \sigma\big(w(\tau) x + b(\tau) \big), 1 \Big)^\top.
	\end{aligned}
	\end{equation}

	\begin{definition}
	Let $T>0$ be fixed. Assume that $F_0 \in \mathcal{P}_1(\R^{d+1})$. We say that the time dependent measure $F_t \in C([0,T]; \mathcal{P}_1(\R^{d+1}))$ is a weak solution to the mean--field equation 
	\begin{equation}\label{mf-mh}
	%\begin{cases}
		\partial_t F_t + \nabla_x \cdot \Big( \sigma\big( w(\tau) x + b(\tau) \big) F_t \Big) + \partial_\tau F_t = 0%\\
		%F(0,x,\tau)=F_0(x,\tau),
	%\end{cases}
	\end{equation}
	with initial condition $F_0$ if for all $\phi = \phi(x,\tau) \in C^\infty_0(\R^{d+1})$ and for all $t\in [0,T]$ the following equality holds:
	\begin{equation} \label{wmf-mh}
	\begin{aligned}
	\int_{\R^{d+1}} \phi(x,\tau) \mathrm{d} F_t(x,\tau) =& 	\int_{\R^{d+1}} \phi(x,\tau) \mathrm{d}F_0(x,\tau) \\
	& + \int_0^t \int_{\R^{d+1}} \nabla_{(x,\tau)}\phi(x,\tau) \cdot G(x,\tau) \mathrm{d}F_s(x,\tau) \mathrm{d}s.
	\end{aligned}
	\end{equation}
	\end{definition}

	Existence and uniqueness of a weak solution $F_t$ of the mean--field equation~\eqref{eq:meso} is obtained under the following assumptions, see~\cite[Section 6.1 and Section 6.2]{tosin} and Proposition~\ref{prop:wellposedness} below:
	\begin{align*}
	\text{(A1)} & \quad \quad \sigma \in C^{0,1}(\R^d), \ w, b \in C^{0,1}(\R);\\
	\text{(A2)} & \quad \quad |\sigma(x) | \leq C_0, \ \forall x \in \R^d. 
	\end{align*}
	
	\begin{remark}
		We observe that Assumption (A2) requires that the activation function $\sigma$ is bounded. This property is verified for some choices of the activation function, e.g.~if $\sigma$ is the hyperbolic tangent function or the sigmoid function, but not in general. However, the results of this section are true if the kinetic measure $F_0$ has compact support, which implies that any $\sigma$ is bounded on the support of $F_0$.
	\end{remark}
	
	\begin{definition} \label{def:flow}
		We define the flow associated to the mean--field equation~\eqref{mf-mh} as the map $\Phi_t:(x,\tau)\in\R^{d+1}\mapsto\Phi_t(x,\tau)\in\R^{d+1}$  such that
		\begin{equation}\label{Phi}
		\begin{cases}
		\partial_t \Phi_t(x,\tau) = G(\Phi_t(x,\tau) ) \\
		\Phi_0(x,\tau) =(x,\tau).
		\end{cases}
		\end{equation}
	\end{definition}

	\begin{proposition} \label{prop:wellposedness}
		Let $F_0\in \mathcal{P}_1(\R^{d+1})$ be given and let $T>0.$ Then, under the assumptions (A1) and (A2), there exists a unique solution $F_t\in C([0,T]; \mathcal{P}_1(\R^{d+1}))$ of the mean--field equation~\eqref{mf-mh}, in particular $F_t = \Phi_t\#F_0$ and $F_t$ is continuously dependent on the initial data $F_0$ with respect to the $1-$Wasserstein distance.Furthermore, the solution of the dynamical system~\eqref{003}--\eqref{001} converges to $F_t$ in $1-$Wasserstein for $M\to\infty$.
	\end{proposition}
	\begin{proof}
		Under the assumptions (A1) and (A2) we have that $G$ defined by equation~\eqref{001} is also Lipschitz and uniformly bounded for all $(x,\tau) \in \R^{d+1}.$ Hence, due to \cite[Lemma 6.1]{tosin}, the flow $\Phi_t$ introduced in Definition~\ref{def:flow}
		is well--defined and Lipschitz in $(x,\tau)$ and $F_t=\Phi_t \# F_0$ for $F_0 \in \mathcal{P}_1(\R^{d+1})$ is the unique weak solution of equation~\eqref{mf-mh} in the sense of~\eqref{wmf-mh}.
		Furthermore, under the assumptions (A1) and (A2), any two weak solutions $F_t^{(1)},F_t^{(2)}$ in the sense of equation~\eqref{wmf-mh}, obtained from initial conditions $F_0^{(1)},F_0^{(2)}$, respectively, fulfill the Dobrushin's stability estimate in $1-$Wasserstein distance. The Dobrushin's inequality allows us to prove the convergence of the solutions of the dynamical system~\eqref{003}--\eqref{001} to $F_t$. In fact, we first observe that if we consider the initial condition
		\begin{equation} \label{initial-mh}
			\de F^M_0(x,\tau)=\frac1M \sum\limits_{i=1}^M \delta(x-x_{i}^0)\delta(\tau)
		\end{equation}
		with $x_i^0$ prescribed by~\eqref{003}, then the following empirical measure
		$$
			\de F^M_t(x,\tau)=\frac1M \sum_{i=1}^M \delta(x - x_i(t))\delta(\tau- \tau_i(t))
		$$
		is a weak solution of~\eqref{mf-mh} in the sense of~\eqref{wmf-mh}, where $(x_i(t),\tau_i(t))$ are the trajectories given by the dynamical system~\eqref{003} for any $i=1,\dots,M$. The previous consideration follows from a classical derivation, see e.g.~\cite{Golse}. Hence, if the initial empirical measure~\eqref{initial-mh} converges in $1-$Wasserstein distance $W_1$ to some $\bar{F}_0 \in \mathcal{P}_1(\R^{d+1})$ for $M\to \infty$, using the Dobrushin's estimate 
		\begin{align*}
			W_1\big( \bar{F}_t,F^M_t \big) \leq C \; W_1\big(\bar{F}_0, F^M_0\big),
		\end{align*}
		with $C$ being a constant and $\bar{F}_t=\Phi_t\#\bar{F}_0$, we obtain that~\eqref{mf-mh} is the mean--field limit of the particle dynamics~\eqref{003} for $M\to\infty$.
	\end{proof}

	The previous proposition shows that the mean--field limit can be obtained provided that the controls $w,b\in C^{0,1}(\R).$ As further result we establish the dependence on the functions $(w,b).$
	
	\begin{proposition}
		Let $F_0\in \mathcal{P}_1(\R^{d+1})$ be given and let $T>0.$ Then, under the assumptions (A1) and (A2), the unique solution $F_t\in C([0,T]; \mathcal{P}_1(\R^{d+1}))$ of the mean--field equation~\eqref{mf-mh} is continuously dependent on $(w,b)$.
	\end{proposition}
	\begin{proof}
		Denote by $\Phi^{(w,b)}$ the flow defined by equation~\eqref{Phi} with given $(w,b)$. Then, for any $(w,b)$ fulfilling (A1) and $\sigma$ fulfilling (A2) the assumptions of \cite[Proposition 7.2]{tosin} are satisfied and we obtain for $F_0 \in \mathcal{P}_1(\R^{d+1})$  
		\begin{align}\label{002}
		W_1( \Phi^{(w,b)}\#F_0, \Phi^{(\bar{w},\bar{b})}\#F_0 ) \leq \frac{ \exp(L t- 1) }L \| (w,b) - (\bar{w},\bar{b}) \|_{C^0(0,T)},
		\end{align}
		where $L=\max\{ L_{G(w,b)}, L_{G(\bar{w},\bar{b})}  \}$ is the maximum of the  Lipschitz constants of $G$ defined by equation~\eqref{001}.
	\end{proof}

		\begin{proposition}
			If a weak solution $F_t\in C([0,T]; \mathcal{P}_1(\R^{d+1}))$ of~\eqref{mf-mh} fulfills
			\begin{equation} \label{eq:factorF}
				\de F_t(x,\tau) = \de f_t(x) \delta(\tau-t)
			\end{equation}
			with $f_t \in C([0,T];\mathcal{P}_1(\R^d))$, and if $\de F_0(x,\tau)=\de f_0(x) \delta(\tau)$ with $f_0 \in \mathcal{P}_1(\R^d)$, then, under the assumptions (A1) and (A2), $f_t$ is a weak solution of the mean--field equation~\eqref{eq:meso} with initial condition $f_0$.
		\end{proposition}
		\begin{proof}
			Using the assumptions on $F_t$ and $F_0$ in~\eqref{wmf-mh} we find for all $\phi=\phi(x) \in C^\infty_0(\R^d)$: 
			\begin{align*}
				\int_{\R^{d}} \phi(x) \mathrm{d}f_t(x) =& 	\int_{\R^{d}} \phi(x) \mathrm{d}f_0(x) \\
				& + \int_0^t \int_{\R^{d}} \nabla_x\phi(x) \cdot \sigma\big( w(t) x + b(t)) \mathrm{d}f_s(x) \mathrm{d}s,
			\end{align*}
			which is exactly the weak form of the mean--field equation~\eqref{eq:meso} with initial condition $f_0$.
		\end{proof}

	\subsection{Mean--field controllability problems} \label{ssec:controllability}
	
	In the continuous formulation of the neural network, the training step can be seen as controllability problem in the sense of the following definition, see also~\cite{Coron,Zuazua}.
	
	\begin{definition} \label{def:meso:controllab}
		Let $f_0,g\in\mathcal{P}_1(\R^d)$ be given. Let $T>0$ be fixed. We say that the mean--field equation~\eqref{eq:meso} is controllable if there exist $w \in C^{0,1}([0,T];\R^{d\times d})$ and $b \in C^{0,1}([0,T];\R^{d})$ such that $(\Phi_T\# f_0)=g$ where $\Phi_t:\R^d\to\R^d$ is the Lipschitz continuous characteristic flow of~\eqref{eq:meso}.
	\end{definition}

%	Observe that the microscopic and the mesoscopic training procedures have the same goal, namely finding optimal $w(t)$ and $b(t)$. However, it is important to notice that the functions $w(t)$ and $b(t)$ in~\eqref{eq:micro} do implicitly depend on the data, i.e.~to be precise one has $w=w(t;M)$ and $b=b(t;M)$. Therefore, the solution of the training process at the microscopic level do not necessarily provide a solution of the mean--field training process.
	
	In this section we focus on the controllability problem at the mean--field level. We show that for simple problems it is possible to recover explicit results on the solution of the controllability problem. However, so far,  a general theory is not available.
	
	\begin{proposition} \label{prop:controllab1}
		Let $T>0$, $\beta>0$ be fixed constants such that $\beta/T$ belongs to the image of $\sigma.$ Further, denote by  $B=(\beta,\dots,\beta)^t\in\R^d$ and  $f_0,g\in\mathcal{P}_1(\R^d)$ be given such that $\de g(x)=\de f_0(x-B)$.
		Then, the  mean--field equation~\eqref{eq:meso} is controllable in the sense of Definition~\ref{def:meso:controllab}.
	\end{proposition}
	\begin{proof}
		According to Definition~\ref{def:meso:controllab} we only need  to show that there exist $w \in C^{0,1}([0,T];\R^{d\times d})$ and $b \in C^{0,1}([0,T];\R^{d})$ such that $g$ is the push--forward of $f_0$ under the flow of~\eqref{eq:meso}. Taking $w(t)\equiv 0$, the flow $\Phi_t:\R^{d}\to\R^{d}$  defined by
		\begin{equation*}
		\partial_t \Phi_t(x) = \sigma(b(t)), \; 
		\Phi_0(x) = x,
		\end{equation*}
		yields the desired result provided that $b$ fulfills $\int_0^T \sigma(b(t)) \mathrm{d}t=B$. There is at least one $b(t)=b_0$ such that the later equality holds.  
%		and one has that
%		$$
%			\de f_T(x) = \de (\Phi_T\# f_0)(x) = \de f_0\left(x-\int_0^T \sigma(b(t)) \mathrm{d}t\right).
%		$$
%		Thus, the system is controllable, i.e.~$f_T=g$, if and only if $\int_0^T \sigma(b(t)) \mathrm{d}t=B$.
	\end{proof}
	
	Clearly, Proposition~\ref{prop:controllab1} does not guarantee uniqueness of the choice of the pair $(w,b)$. Consider, e.g., $d=1$ and the identity activation function $\sigma(x)=x$. Then for any $\beta>0$ the training process can be solved with $b(t)=t^2+1$. Namely, there exists a time $T$ at which $f_T=g$. However, for the same time $T$, the training process is also solved with $b(t)=\beta/T$.
	
	The next proposition is recalled  from \cite[Proposition 1]{HTV_MFResNet} and it characterizes  steady states of the mean--field equation that are not necessarily unique. Hence, we do not expect the controllability problem with infinite time horizon to solve explicitly. The proposition only illustrates that if initial and terminal states are a sum of weighted Dirac measures the system is trivially controllable with parameters $(\overline{w}, \overline{b})$ given below. However, as the proposition shows this only possible if the activation function as sufficiently many zeros.
	
	\begin{proposition}
		Let $f_t\in\mathcal{P}_1(\R^d)$ be the compactly supported weak solution of the mean--field equation~\eqref{eq:meso}. Assume that the activation function $\sigma:\R\to\R$ have $n$ zeros $z_i$, i.e.~$\sigma(z_i)=0$ for $i=1,\dots,n$. Let $\bar{b} = \lim_{t\to\infty} b(t) \in \R^d$ and $\bar{w}= \lim_{t\to\infty} w(t) \in \R^{d\times d}$ exist and be finite. Moreover, assume that $\bar{w}$ has maximum rank. Then,
		\begin{equation*} \label{eq:weakSS}
		f_{\infty} = \sum_{i=1}^{n^d} \rho_i \delta_{y_i}
		\end{equation*}
		is a steady state solution of~\eqref{eq:meso} in the sense of measures provided that $y$ is the solution to the system $\bar{w} y + \bar{b} = z$, where $z$ is any disposition with repetition of the $n$ zeros, and where $\rho_i\in[0,1]$, $\forall\,i=1,\dots,n$, with $\sum_{i=1}^{n^d} \rho_i=1$.
	\end{proposition}

	The next lemma shows that for particular choices of $\sigma,w,b$ the mean of $f_t$ is preserved.
	\begin{lemma} \label{lem:mconserved}
		Let $f_0\in\mathcal{P}_1(\R^d)$ and let $f_t\in\mathcal{P}_1(\R^d)$ be a solution of the mean--field equation~\eqref{eq:meso} where $m(t) = \int_{\R^d} x \de f_t(x)$, $\forall t\in[0,T]$. Assume that $\sigma(x)=x$ and $b(t)=-w(t)m(t)$, $\forall t\in[0,T]$. Then, $m(t)=m(0)$, $\forall t\in[0,T]$.
	\end{lemma}
	\begin{proof}
		Using equation~\eqref{eq:meso}  the first moment $m(t)$ of the probability density $f_f$ satisfies the evolution equation
		\begin{align*}
		\frac{\de}{\de t} m(t) = w(t) m(t) + b(t).
		\end{align*}
		Taking $b(t)=-w(t)m(t)$ the right--hand side vanishes and the assertion follows.
	\end{proof}
	
	This lemma can be used to obtain a further controllability result in Proposition~\ref{prop:controllab2} below. However, we restrict the case to $d=1$ since for $d>1$ one has to assume additional assumptions on the matrix $w(t)$ in order to write the solution formula of the flow explicitly.
	
	\begin{proposition} \label{prop:controllab2}
		Let $T>0$, $\alpha\in\R$ be fixed constants. Let $f_0,g\in\mathcal{P}_1(\R)$ be given such that $g=(F^{-1}\# f_0)e^\alpha$ where $F:x\in\R\mapsto F(x)=x e^\alpha + (1-e^\alpha)m_0\in\R$ and $m_0 = \int_{\R} x \de f_0(x)$. Assume that $\sigma(x)=x$.
		Then the mean--field equation~\eqref{eq:meso} is controllable in the sense of Definition~\ref{def:meso:controllab}.
	\end{proposition}
	\begin{proof}
		Observe that with the definition of $g$ we have
		$$
		\int_{\R} x \de g(x) = \int_{\R} e^\alpha x \de (F^{-1}\#f_0)(x) = \int_{\R} e^\alpha \left( x \circ F^{-1} \right)(x) |\text{det}J_{F^{-1}}| \de f_0(x) = m_0.
		$$
		In order to have $m(t)=m_0$ for all times we set $b(t) = -w(t) m_0$ and $\sigma(x)=x$. Further, we choose $w \in C^{0,1}([0,T];\R^{d\times d})$ such that $\int_0^T w(t) \de t= -\alpha$. The flow $\Phi_t:\R\to\R$ is defined by
		\begin{equation*}
		\partial_t \Phi_t(x) = w(t) \left( \Phi_t(x) - m_0 \right), \; 		\Phi_0(x) = x,
		\end{equation*}
		which yields $\Phi_t(x)= e^{\int_0^t w(s) \de s} x - m_0 \int_0^t e^{\int_r^t w(s) \de s} \de r$ and $\Phi_T^{-1}(y)=e^\alpha \left( y + m_0 
		\int_0^T e^{\int_r^T w(s) \de s} \de r \right).$ Hence,   
		$$
		\de f_T(x) = \de (\Phi_T\# f_0)(x) = \de f_0\left(x e^{-\int_0^T w(t) \de t} + (1-e^{-\int_0^T w(t) \de t})m_0\right)e^{\alpha}.
		$$
		and the system is controllable, i.e.~$f_T=g$.
	\end{proof}
	
	Proposition~\ref{prop:controllab2} shows that the sign of $\alpha$ influences the behavior of the second moment $\int x^2 \de f_t(x)$ of $f_t$. In particular, for $\alpha>0$ it is possible to show that the second moment of $f_t$ decreases and $f_t$ concentrates at the first moment, whereas for $\alpha<0$ the second moment increases.

	Proposition~\ref{prop:controllab1} and Proposition~\ref{prop:controllab2} of this section discuss some prototype situations in which the problem of recovering the target distribution $g$ through the mean--field equation~\eqref{eq:meso} with initial condition $f_0$ is explicitly possible. In general,  we follow an alternative method introduced in the subsequent sections to compute the  parameters.
	
%	\subsection{Mean--field limit of the loss function} \label{ssec:mfLoss}
	\subsection{Existence of solutions to the mean--field minimization problem} \label{ssec:existenceMin}

	As presented in Section~\ref{sec:preliminary} the training procedure of a neural network aims to find optimal weights and bias in order to minimize a given distance, e.g.~\eqref{eq:lossMicro}. The mean--field interpretation of this training procedure requires first to derive the mean--field limit of the loss function which is given by the following proposition.
	\begin{proposition}
		Let $\{ x_i(t),\tau_i(t) \}_{i=1}^M$ be the trajectories given by the dynamical system~\eqref{003}-\eqref{001} with initial conditions $\{ x_i^0,0 \}_{i=1}^M$ and let $\{y_i\}_{i=1}^M$ be the given target values. Let $F_0^M\in\mathcal{P}_1(\R^{d+1})$ and $g^M\in\mathcal{P}_1(\R^{d})$ be the empirical measures associated to the initial conditions and the target values, respectively. Furthermore, let $F_0\in\mathcal{P}_1(\R^{d+1})$ and $g\in\mathcal{P}_1(\R^{d})$ be such that $W_1(F_0,F_0^M) \to 0$, $W_1(g,g^M) \to 0$, as $M\to\infty$. Then, under the assumptions (A1) and (A2), the mean--field limit of the loss function~\eqref{eq:lossMicro} is
		$$
			\int_{(x,\tau)\in\R^{d+1}} \tilde{\ell}(x) \de F_T(x,\tau)
		$$
		where $\tilde{\ell}(x) = \int_{\R^{d}} \ell(x-y) \de g(y)$ and $F_t$ is the weak solution of~\eqref{mf-mh} obtained with initial condition $F_0\in\mathcal{P}_1(\R^{d+1})$.
	\end{proposition}
	\begin{proof}
		We notice that the loss function~\eqref{eq:lossMicro} can be written as
		$$
		\frac{1}{M} \sum_{i=1}^M \ell(x_i(T)-y_i) = \int_{\R^{2d+1}} \ell(x-y) \de \mu^M_T(x,y,\tau)
		$$
		where $\mu^M_t\in\mathcal{P}_1(\R^{2d+1})$ is the time dependent empirical measure
		\begin{equation} \label{eq:jointEmpirical}
		\de\mu^M_t(x,y,\tau) = \frac{1}{M} \sum_{i=1}^M \delta(x-x_i(t)) \delta(y-y_i) \delta(\tau-t) 
		\end{equation}
		We observe that $\mu^M_T$ has marginals
		\begin{align*}
		\int_{(x,\tau)\in\R^{d+1}} \de\mu^M_T(x,y,\tau) &= \frac{1}{M} \sum_{i=1}^M \delta(y-y_i) = \de g^M(y), \\
		\int_{y\in\R^d} \de \mu^M_T(x,y,\tau) &= \frac{1}{M} \sum_{i=1}^M \delta(x-x_i(T)) \delta(\tau-t) = \de F^M_T(x,\tau).
		\end{align*}
		In particular, since predictions and target values are statistically independent, we have
		$$
		\de \mu^M_T(x,y,\tau) = \de F^M_T(x,\tau) \de g^M(y).
		$$
		By the Glivenko--Cantelli's theorem, see e.g.~\cite{FournierGuillin,Boissard}, we have that there exists $g\in\mathcal{P}(\R^d)$ and a time dependent measure $\mu_t\in\mathcal{P}(\R^{2d+1})$ such that $W_1(g^M,g) \to 0$ and $W_1(\mu^M_t,\mu_t) \to 0$, as $M\to\infty$. In addition, under the assumptions (A1) and (A2), the mean--field convergence result of Section~\ref{ssec:wellposed} implies $W_1(F^M_t,F_t) \to 0$, $\forall t\in[0,T]$ as $M\to\infty$, with $F_t\in\mathcal{P}_1(\R^{d+1})$ weak solution of~\eqref{mf-mh} obtained with initial condition $F_0\in\mathcal{P}_1(\R^{d+1})$ such that $W_1(F_0,F_0^M) \to 0$, as $M\to\infty$. Then, noticing that
		$$
		W_1(\mu^M_T,F_T g) \leq W_1(F^M_T g^M,F^M_T g) + W_1(F^M_T g,F_T g) \to 0 \ \text{ as } M\to\infty,
		$$
		we conclude that the mean--field limit of the loss function~\eqref{eq:lossMicro} is
		\begin{equation} \label{eq:lossMF}
		\int_{(x,y,\tau)\in\R^{2d+1}} \ell(x-y) \de \mu_T(x,y,\tau) = \int_{(x,\tau)\in\R^{d+1}} \tilde{\ell}(x) \de F_T(x,\tau), \quad \tilde{\ell}(x) = \int_{\R^{d}} \ell(x-y) \de g(y).
		\end{equation}
	\end{proof}
	
%	\begin{remark}
%	As possible interpretation of the loss function in the mean--field context,  we consider its relation to the $1$--Wasserstein distance. Consider the  empirical measure $\mu^M_t$ given by ~\eqref{eq:jointEmpirical}. Then, there exists $\mu_t\in\mathcal{P}_1(\R^{2d+1})$ such that $W_1(\mu_t^M,\mu_t) \to 0$ as $M\to\infty$. Moreover, there exists $\pi_t\in\mathcal{P}_1(\R^{2d})$ such that $\mu_t = \pi_t \delta(\tau-t)$ and $\pi_t(x,y) = \int_{\tau\in\R} \de \mu_t(x,y,\tau)$. For the same reason, we have $F_t = f_t \delta(\tau-t)$ where $f_t(x) = \int_{\tau\in\R} \de F_t(x,\tau)$ is the weak solution of the mean--field of the original dynamics. Thus, $\de\pi(x,y)_t=\de f_t(x) \de g(y)$ is an admissible transport plan between $f_t$ and $g$, i.e.~$\pi_t\in\Pi(f_t,g)$ and
%	$$
%		\int_{\R^{2d+1}} \ell(x-y) \de\mu_T(x,y,\tau) = \int_{\R^{2d}} \ell(x-y) \de \mu_T(x,y).
%	$$
%	Taking $\ell(x-y)=|x-y|$ we obtain 
%	$$
%		W_1(f_t,g) \textcolor{red}{= \min\limits_{\pi\in \Pi(\mu, \nu)} \int_{\R^{2d}} |x-y| \de\pi(x,y)} \leq \int_{\R^{2d}} \ell(x-y) \de \mu_T(x,y).
%	$$
%	Hence,  the loss function can be seen as an upper bound for the $1$--Wasserstein distance between $f_T$ and $g$.
%	\end{remark}
	
%	\subsection{Existence of solutions to a minimization problem} \label{ssec:existenceMin}
	
	Consider now the cost functional $J:\mathcal{P}_1(\R^{d+1}) \to \R$ given by 
	\begin{align}\label{cost-mh}
		J(\mu) = \int_{\R^{d+1}} \tilde{\ell}(x) \mathrm{d}\mu(x,\tau).
	\end{align}
	In the following we discuss the existence of solutions to the mean--field minimization problem
	$$(w,b) \mapsto \min J(F_T)$$
	on a suitable subset $X$ of controls $(w,b)$, where $F_T$ is the unique weak solution to equation \eqref{mf-mh} for fixed initial datum $F_0 \in \mathcal{P}_1(\R^{d+1})$. We observe that the mean--field cost functional $\tilde{\ell}$ derived in~\eqref{eq:lossMF} is bounded and Lipschitz continuous provided that $\ell \in C^{0,1}(\R^d)$ is bounded from below. In fact:
	$$
		\| \tilde{\ell}(x) - \tilde{\ell}(z) \| \leq \int_{\R^{d}} \| \ell(x-y) - \ell(z-y) \| \de g(y) \leq L \int_{\R^{d}} \| x-z \| \de g(y) = L \| x-z \|, \quad \forall x,z\in\R^d,
	$$
	with $L$ Lipschitz constant of $\ell$.
	In order to simplify the notation we denote by
	\begin{align*} u&:=(w,b) \in C^{0,1}([0,T];\R^{d\times d}\times\R^d) , \\ \mu_u&:=F_T \in \mathcal{P}_1(\R^{d+1}). 
	\end{align*}
	For fixed $T>0$ and $F_0\in \mathcal{P}_1(\R^{d+1})$ the reduced cost functional 
	is then defined by 
	\begin{align*}
	j(u)=J(\mu_u)=\int_{\R^{d+1}} \tilde{\ell}(x) \mathrm{d}\mu_u(x,\tau).%: X \to \R.
	\end{align*}
	Since $\tilde{\ell}$ is bounded from below, we obtain that $j$ is bounded from below. Since $\tilde{\ell}$ is Lipschitz with constant $L$ we obtain the following estimate for $u,v \in C^{0,1}([0,T];\R^{d\times d}\times\R^d)$ constrained to $u(0)=v(0)=0$:
	\begin{align*}
	|j(u)-j(v)| &= L \left\| \int_{\R^{d+1}} \frac{\tilde{\ell}(x)}{L}  \mathrm{d}(\mu_u-\mu_v)  \right\|	 \\ & 
	\leq L  \sup\left\{  \int_{\R^{d+1}} \phi(x)  \mathrm{d}(\mu_u-\mu_v) : \phi \mbox{ 1--Lipschitz} \right\} \\  &= L W_1(\mu_u, \mu_v) 
	\leq C(L_u,L_v) \| u - v \|_{C^0}, 
	\end{align*}
	where the last inequality follows by \eqref{002} and the constant $C$ depends on the Lipschitz constants of $u$ and $v.$ Thus, the loss function $j$ is continuous with respect to the $C^0-$norm.
	The previous results can be used to establish existence of minimizers using the direct method of variation.
	\begin{proposition}
		Assume that the assumptions (A1) and (A2) are fulfilled. Let $\tilde{\ell} \in C^{0,1}(\R^{d})$ 
		and bounded from below. Assume $T>0, L>0$ and $F_0\in \mathcal{P}_1(\R^{d+1})$ are given.
		Then, there exists a solution to the  minimization problem 
		\begin{align*} \min\limits_{ (w,b) \in X } \int_{\R^{d+1}} \tilde{\ell}(x) \mathrm{d}F_T(x,\tau),  
		\end{align*}		  
		where $F_t\in C([0,T]; \mathcal{P}_1(\R^{d+1}))$ is the weak solution to equation \eqref{mf-mh} and where
		\begin{equation} \label{eq:X}
		X=\{ (w,b) \in C^{0,1}([0,T];\R^{d\times d}\times\R^d): L_w + L_b \leq L, \; w(0)=b(0)=0 \}
		\end{equation}
		with $L_w,L_b$ Lipschitz constants of $w,b$, respectively.
	\end{proposition}
	\begin{proof}
		For the proof of the previous proposition we proceed as follows. Since the cost functional is bounded from below, there exists a minimizing sequence $(u_n)_{n\geq0} \subset X.$ According to the definition of $X$ we have that $\|u_n\|_{C^0} \leq L$ for all $n.$ Further, we have that $u_n$ is uniformly Lipschitz continuous due to definition of $X.$ Hence, the assertion of the Arzela--Ascoli are fulfilled and $u_n$ converges in $C^0$ to $u \in C^0([0,T];\R^{d\times d}\times\R^d).$ Furthermore, it holds that $u \in C^{0,1}([0,T];\R^{d\times d}\times\R^d)$ with Lipschitz constant bounded by $L.$ Finally, the continuity of $j$ shown above yields that $u$ is the minimizer, i.e., $j(u)=\lim\limits_{n\to \infty} j(u_n).$ This finishes the proof.
	\end{proof}

	\section{Computational approach to the mean--field training procedure} \label{sec:training}
	
	In this section we explicitly formulate the training processes for the mean--field limit~\eqref{eq:meso} of the neural differential equation~\eqref{eq:micro} in terms of an optimal control problem. While Section~\ref{ssec:existenceMin} discusses existence of minimizers to the mean--field control problem, here, we formally derive a first--order optimality system in order to design a numerical method for the optimization of the parameters $w(t)$ and $b(t)$ of the neural network. Note that the previous theorem does not allow for a characterization due to a lack of regularity of the solution in terms of the parameters.

	\subsection{Formulation of the computational approach}
	
	Our computational approach is based on the minimization problem of a general functional of equation \eqref{cost-mh} for $(w,b) \in X$, cf.~\eqref{eq:X}: 
	\begin{equation*}\label{mh-01}
	\min\limits_{ (w,b) \in X } \int_{\R^{d+1}} \tilde{\ell}(x) \de F_T(x,\tau) + \frac{\gamma}2 \int_0^T \| w \|^2 + \|b\|^2 \de t,
	\end{equation*}
	where $F_T$  is the solution at some time $t=T$  obtained as (weak) solution of~\eqref{mf-mh} endowed with the initial condition $F_0$.
	We assume the initial and solution are of the type~\eqref{eq:factorF} $\forall (t,x,\tau)$. 
	Furthermore, we use a Tikhonov regularization for the controls with a parameter $\gamma>0$ but the proposed computational approach below works also in the case $\gamma=0$. Under the structural assumption~\eqref{eq:factorF} the previous problem reduces to a constrained optimal control problem 
	\begin{align*}
	&\min\limits_{ (w,b) \in X } \int_{\R^{d}} \tilde{\ell}(x) \de f_T(x) + \frac{\gamma}2 \int_0^T \| w \|^2 + \|b\|^2 \de t  \\
	& \mbox{ subject to } \ 
	\begin{cases}
	\partial_t f_t(x) + \nabla_x \cdot \Big( \sigma\big( 
	w(t) x+ b(t) 
	\big) f_t(x) \Big) =0, \\
	f_{t=0}(x)=f_0(x).
	\end{cases}
	\end{align*}
	Only formally, a first--order optimality system in strong form is derived
	\begin{subequations}
	\begin{align}
	\partial_t f_t(x) + \nabla_x \cdot \Big( \sigma\big( 
	w(t) x+ b(t) 
	\big) f_t(x) \Big) =0, &\quad f_{t=0}(x)=f_0(x), \\
	\partial_t \Lambda_t(x) + \nabla_x \Lambda_t(x) \cdot \sigma\big( 
	w(t) x+ b(t) 
	\big) = 0, &\quad \Lambda_{t=T}(x) = \tilde{\ell}(x), \label{mh-adj} \\
	\gamma b_j(t) + \int_{\R^d} \partial_{x_j}\Lambda_t(x) \sigma'_j\big( 
	w(t) x+ b(t) 
	\big) f_t(x) \de x =0, & \quad j=1,\dots,d,  \\
	\gamma w_{j,k}(t) + \int_{\R^d} \partial_{x_j} \Lambda_t(x) \sigma'_j\big( 
	w(t) x+ b(t) 
	\big) x_k f_t(x) \de x =0, & \quad j,k=1,\dots,d, 
	\end{align}
	\end{subequations}
	where $b_j$ represents the $j$--th component of the bias vector $b$, $w_{j,k}$ is the entry $(j,k)$ of the weight matrix $w$ and, finally, $\sigma'_j$ represents the derivative of the activation function $\sigma$ computed on the $j$--th component of its argument.
	Note that the constraint $w(0)=0 \in \mathbb{R}^{d\times d}$ and $b(0)=0\in\mathbb{R}^d$ are enforced
	in the numerical method. We further ormally differentiate the equation for $\Lambda_t$ with respect to $x_j$ for $j=1,\dots,d.$ This also yields a conservative formulation for each $\partial_{x_j} \Lambda_t.$ Furthermore, we transform time $t\mapsto T-t$ in equation \eqref{mh-adj} in order to obtain an initial value problem. Since we implement numerical results in the case $d=1$ we state the resulting system where $\lambda_t(x)=\partial_x \Lambda_{T-t}(x)$
	\begin{subequations}
	\begin{align}
	\partial_t f_t(x) + \partial_x \Big( \sigma\big( 
	w(t) x+ b(t) 
	\big) f_t(x) \Big) =0, &\quad f_{t=0}(x)=f_0(x), \label{mh-0001} \\
	\partial_t \lambda_t(x) - \partial_x \Big( \sigma\big( 
	w(T-t) x+ b(T-t) 
	\big) \lambda_t(x) \Big) = 0, &\quad \lambda_{t=0}(x) = \partial_x \tilde{\ell}(x), \label{mh-0002} \\
	\gamma b(t) + \int_{\R} \lambda_{T-t}(x) \sigma'\big( 
	w(t) x+ b(t) 
	\big) f_t(x) \de x =0, &\; \label{mh-0003}  \\
	\gamma w(t) + \int_{\R}  \lambda_{T-t}(x) \sigma'\big( 
	w(t) x+ b(t) 
	\big) x f_t(x) \de x =0. & \label{mh-0004}
	\end{align}
	\end{subequations}
	Observe that~\eqref{mh-0001} and~\eqref{mh-0002} are decoupled due to the definition of the loss function, namely the initial state of~\eqref{mh-0002} does not depend on the final state of~\eqref{mh-0001}. This is due to the fact that the mean--field loss function is linear in the state. 

	\subsubsection{Numerical discretization scheme}

	The optimality system is solved in a block Gauss--Seidel fashion, i.e., at the $k$--th iteration $t\mapsto (w,b)^k(t)$ we compute $f^k_t$ and $\lambda^k_t$ as numerical solution to equation~\eqref{mh-0001} and~\eqref{mh-0002}, respectively. Details on the numerical scheme will be presented below.
	The new iterates $(w,b)^{k+1}$ are found through iteration on equations \eqref{mh-0003}--\eqref{mh-0004}, namely for $t > 0$ we define
	\begin{align*}
	b^{*}(t;\rho) = b^k(t) - \rho^* \left(\gamma b^k(t) + \int_{\R} \lambda_{T-t}(x) \sigma'\big( 
	w^k(t) x+ b^k(t) 
	\big) f_t(x) \de x  \right),
	\end{align*}
	$b^{*}(0)=0,$ and similarly  $w^*.$ Here, $\rho^*>0$ is a stepsize parameter chosen using backtracking line search, e.g.~the Armijo rule, in order to minimize the reduced cost functional 
	\begin{align*}
	\rho^* = \argmin\limits_{\rho>0} \left( 	\int_{\R} \tilde{\ell}(x) \de f_T(x;\rho) + \frac{\gamma}2 \int_0^T \| w^{*}(t;\rho) \|^2 + \|b^{*}(t;\rho) \|^2 \de t \right)
	\end{align*}
	where $f_T(x;\rho)$ is the solution to~\eqref{mh-0001} for $(w,b)=(w^*(t;\rho),b^*(t;\rho)).$ 
	The new iterates are then obtained by 
	\begin{align*}
	(w,b)^{k+1}(t) := (w^*,b^*)(t;\rho^*), \ \forall t\geq0.
	\end{align*}
	The procedure is repeated $k=0,1,\dots$  until the error $e^{(k)}$ is below a given tolerance $TOL$:
	\begin{align*}
	e^{(k+1)} := \frac{\| (w,b)^{k+1} - (w,b)^k \|_{C^0(0,T)}}{\| (w,b)^{k+1}\|_{C^0(0,T)}} \leq TOL. 
	\end{align*}
	For more details on the iterative scheme described above we refer, e.g., to~\cite{NocedalWright}.
	\begin{remark}
	In the case $\sigma(x)=x$ further simplifications 
	are possible. In fact, we may derive explicit equations for the evolution of the 
	$k$--th moment of $\lambda \; f$ given by 
	\begin{align*}
	\partial_t \int_{\R} x^k \lambda_{T-t}(x) f_t(x) \de x = - (k+1) w(t) \int_{\R} x^k \lambda_{T-t}(x) f_t(x) \de x. 
	\end{align*}
	This allows to obtain $(w,b)$ in closed form
	\begin{align}\label{mh-0005}
	b(t) &= \frac{1}\gamma \exp( w(t) ) \int_{\R} \lambda_T(x) f_0(x) \de x, \\
	w(t) \exp( - 2 w(t) ) &= \frac{1}\gamma  \int_{\R} x \lambda_T(x) f_0(x) \de x. 	\label{mh-0006}
	\end{align}
	In this case it is sufficient to iterate equations~\eqref{mh-0002} and equations~\eqref{mh-0005}--\eqref{mh-0006} removing the need to solve equation~\eqref{mh-0001}.
	\end{remark}
	
	The numerical solution of the PDEs~\eqref{mh-0001} and~\eqref{mh-0002} is computed with a third--order finite volume scheme~\cite{CPSV:cweno}, which is briefly described below. Both equations are recast in the following compact formulation:
	\begin{equation} \label{eq:generalPDE}
	\partial_t u(t,x) + \partial_x \mathcal{L}(u(t,x),t,x) = 0,
	\end{equation}
	with $\mathcal{L}$ linear operator with respect to $u$. Since the two PDEs are decoupled, they can be solved simultaneously. Application of the method of lines to~\eqref{eq:generalPDE} on discrete cells $\Omega_{j}$, defining a discretization of the physical domain $\Omega$, leads to the coupled system of ODEs
	\begin{equation} \label{eq:semidiscrete}
	\frac{\mathrm{d}}{\mathrm{d}t} \overline{U}_j(t) = - \frac{1}{\Delta x} \left[ \mathcal{F}_{j+\frac12}(t) - \mathcal{F}_{j-\frac12}(t) \right],
	\end{equation}
	where $\overline{U}_{j}(t)$
	% $$
	% \overline{U}_{j}(t) \approx \frac{1}{\Delta x} \int_{\Omega_{j}} u(t,x) \mathrm{d}x
	% $$
	is the approximation of the cell average of the exact solution $u$ in the cell $\Omega_{j}$ at time $t$.
	Here, $\mathcal{F}_{j+\frac12}(t)$ approximates $\mathcal{L}(u(t,x_{j+\frac12}),t,x_{j+\frac12})$ with suitable accuracy and is computed as a function of the boundary extrapolated data $U_{j+\frac12}^{\pm}(t)$, i.e.
	$$
	\mathcal{F}_{j+\frac12}(t) = \mathcal{F}(U_{j+\frac12}^{+}(t),U_{j+\frac12}^{-}(t))
	$$
	and $\mathcal{F}$ is a consistent and monotone numerical flux, evaluated on two estimates of the solution at the cell interface. We focus on the class of central schemes, in particular we consider $\mathcal{F}$ as a local Lax--Friedrichs flux.
	In order to construct a third--order scheme the values $U_{j+\nicefrac12}^{\pm}(t)$ at the cell boundaries are computed with the third--order CWENO reconstruction~\cite{CPSV:cweno}.
	
	System~\eqref{eq:generalPDE} is finally solved by the classical third--order (strong stability preserving) SSP Runge--Kutta with three stages~\cite{JiangShu:96}. At each Runge--Kutta stage, the cell averages are used to compute the reconstructions via the CWENO procedure and the boundary extrapolated data are fed into the Lax--Friedrichs numerical flux. The initial data are computed with the three point Gaussian quadrature. The time step $\Delta t$ is chosen fixed in order to have a fixed grid in time and to avoid a reconstruction in time of the control functions $w(t)$ and $b(t)$ between different iterates of the Gauss--Seidel approach. All the simulations are run with a CFL of $0.45$. The other parameters of the simulations are specified in each numerical example separately.
	
	\subsection{Computational results}
	
	We present three numerical experiments in order to illustrate the numerical solution of the training of the mean--field neural network and to numerically observe the theoretical findings on the controllability approach presented in Section~\ref{ssec:controllability}.
	
	As loss function we use
	\begin{equation} \label{eq:lossnumerics}
		\ell(x-y)  = | x - y |^2.
	\end{equation}
	and  the initial condition of equation~\eqref{mh-0002} is then given by 
	$$
		\lambda_0(x) = 2 x \int_\R \de g(y) - 2 \int_\R y \de g(y) = 2 x - 2 m_g
	$$
	where $m_g$ denotes the expected value of the target $g$.
	
	\begin{remark} \label{rem:loss}
		The initial condition of the equation for $\lambda$ depends only on the expected value of the target. Hence, if we consider two different targets $g_1$ and $g_2$ such that $m_{g_1}\neq m_{g_2}$, we expect to be able to recover correctly the expected value only. For example, an $L^2_x$ distance between the final state $f_T$ and the target $g$ would also lead to a dependence on the full state. However, the formal mean--field of the discrete loss function does not include this choice as shown in the previous section 
	\end{remark}
	
	\paragraph{Test 1.} In the first example we provide a numerical evidence of the controllability problem proposed in Proposition~\ref{prop:controllab1}. We choose the initial condition
	$$
		f_0(x) = \chi_{[-\frac12,\frac12]}(x)
	$$
	on the physical domain $x\in\Omega=[-2,3]$, and the target is
	$$
		g(x) = f_0(x-\beta)
	$$
	with $\beta=1$. The time at which we aim to recover the target $g$ is $T=1$. We consider a fixed time step $\Delta t = 10^{-2}$ and the space domain is discretized with $200$ cells. The regularization parameter is $\gamma = 10^{-3}$ and the tolerance for the stopping criterion is $TOL = 10^{-4}$. The maximum number of iteration of the Armijo--stepsize rule is $10$. Three different activation functions are considered, $\sigma(x)=x$, $\sigma(x)=\tanh(x)$ and $\sigma(x)=\frac{1}{1+\exp(-x)}$. The initial guess of the controls is $w^0(t)=0$ and $b^0(t) = 0$, $\forall\,t\in[0,1]$. 
	
	\begin{figure}[t!]
		\centering
		\includegraphics[width=\textwidth]{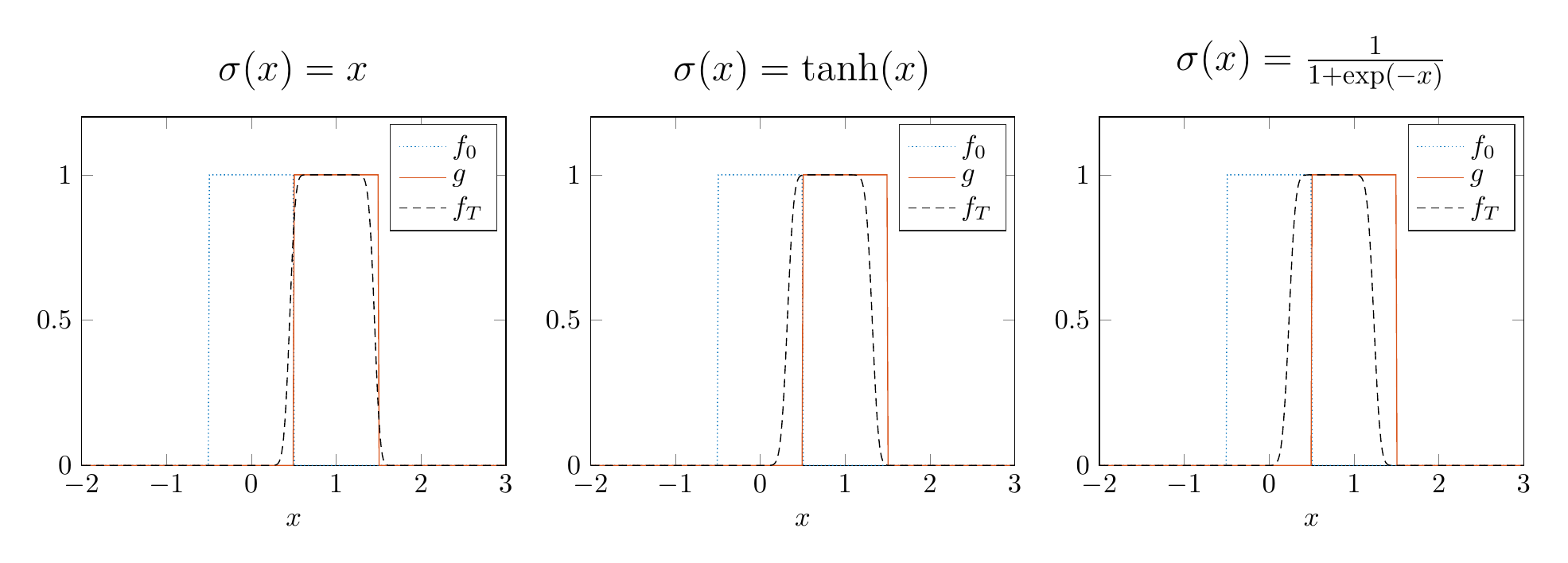}
		\caption{Final states $f_T$ (black dashed lines) at time $T=1$ obtained with the optimal controls $w(t)$ and $b(t)$ computed by the block Gauss--Seidel approach with tolerance $TOL=10^{-4}$. Three different activation functions are considered and specified in the panel titles. The blue dotted lines represent the initial state $f_0$, whereas the red solid lines represent the target $g$.\label{fig:control1_sol}}
	\end{figure}
	
	\begin{figure}[t!]
		\centering
		\includegraphics[width=\textwidth]{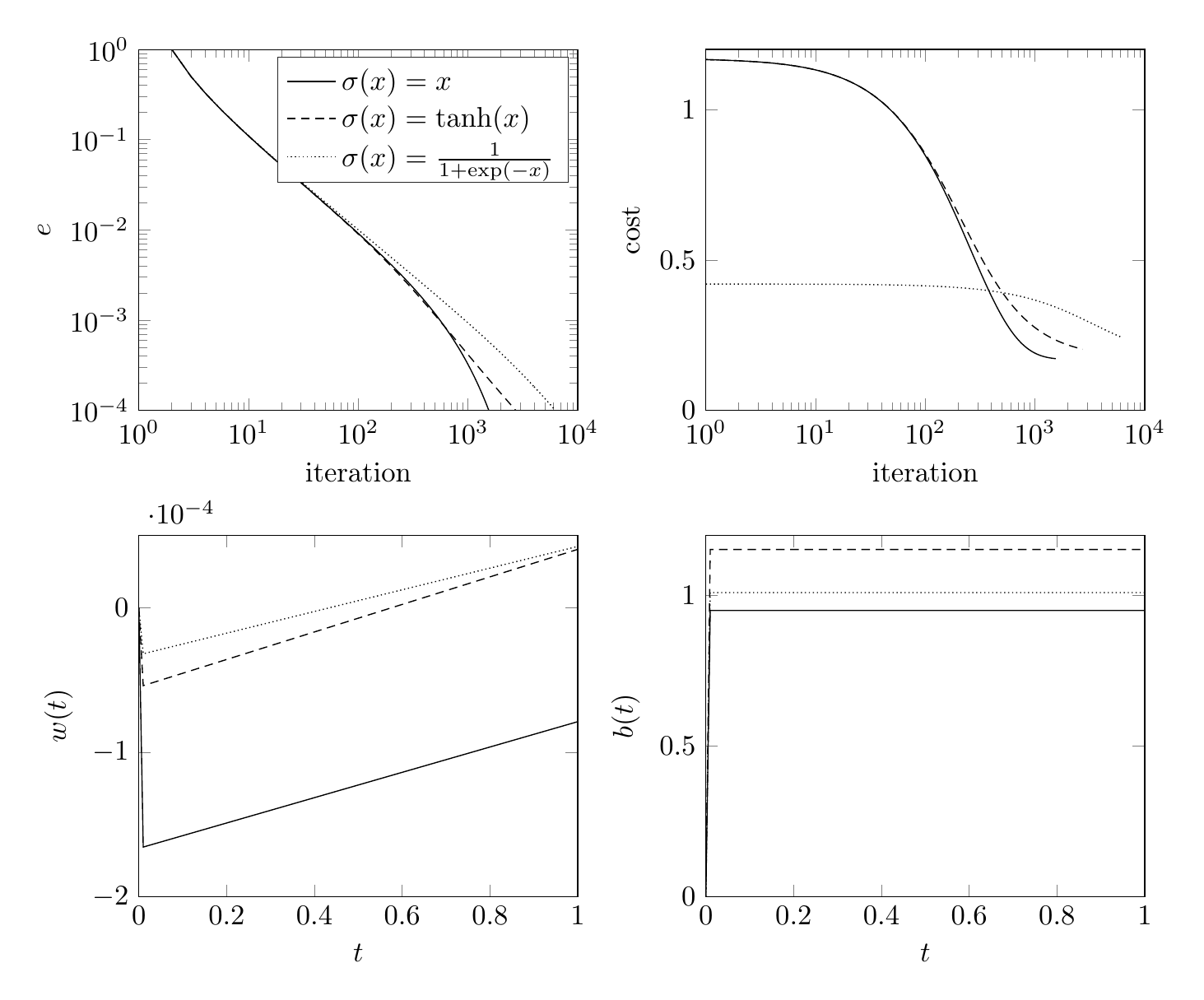}
		\caption{Top left panel: behavior of the relative errors between two consecutive iterations of the controls. Top right panel: value of the cost functionals at each iteration. Bottom panels: optimal controls $w(t)$ (left) and $b(t)$ (right). In all plots, the solid line represent the case of the identity activation function, the dashed line the hyperbolic tangent, and the dotted line the sigmoid.\label{fig:control1_analysis}}
	\end{figure}
	
	In Figure~\ref{fig:control1_sol} we compare the final states $f_T$ (black dashed lines) obtained with the three different activation functions and the optimal controls $w(t)$, $b(t)$ which are shown in the bottom panels of Figure~\ref{fig:control1_analysis}. We observe that using the identity activation function provides a better approximation of the target $g$, whereas for the other two activation functions additional iterations of the optimization procedure are required. In fact, the top panels of Figure~\ref{fig:control1_analysis} shows that, while the relative error between two iterates of the controls reaches the given tolerance $TOL$ for all the activation functions, the values of the cost functional for the case of the hyperbolic tangent and of the sigmoid are larger than the value of the cost functional obtained with the identity and are still decreasing towards a minimum value. Furthermore, we observe that the initial guess of the controls is a better choice for the sigmoid activation. The optimal controls are depicted in the bottom panels of Figure~\ref{fig:control1_analysis} and they are $w(t)\approx 0$ and $b(t)=C$ with $C$ positive constant which depends on the choice of the activation function. Observe, in particular, that $C\approx1$ for the identity activation, which means that equation~\eqref{mh-0001} reduces
	$$
		\partial_t f_t(x) + C \partial_x f_t(x) = 0
	$$
	whose solution at $T=1$ is $f_T(x) = f_0(x-C) \approx g(x) = f_0(x-1)$. This result is consistent with Proposition~\ref{prop:controllab1}.
	
	\paragraph{Test 2.} In the second example we provide a numerical evidence of the controllability problem proposed in Proposition~\ref{prop:controllab2}. We choose the initial condition
	$$
		f_0(x) = \frac{1}{\sqrt{2\pi s^2}} e^{-\frac{(x-\mu)^2}{2s^2}}
	$$
	on the physical domain $x\in\Omega=[-2,3]$ with $s=0.1$ and $\mu=1$. The target is
	$$
		g(x) = f_0(xe^\alpha+(1-e^\alpha)\mu)e^\alpha
	$$
	with $\alpha=0.25$. The time at which we aim to recover the target $g$ is $T=1$. We consider a fixed time step $\Delta t = 10^{-2}$ and the space domain is discretized with $400$ cells. The regularization parameter is $\gamma = 10^{-3}$ and the tolerance for the stopping criterion is $TOL = 10^{-4}$. The maximum number of iteration of the Armijo--stepsize rule is $10$. Due to Lemma~\ref{lem:mconserved} and Proposition~\ref{prop:controllab2} we take the identity activation functions $\sigma(x)=x$. The initial guess of the controls is $w(t)=0$ and $b(t) = 0$, $\forall\,t\in[0,1]$.
	
	\begin{figure}[t!]
		\centering
		\includegraphics[width=\textwidth]{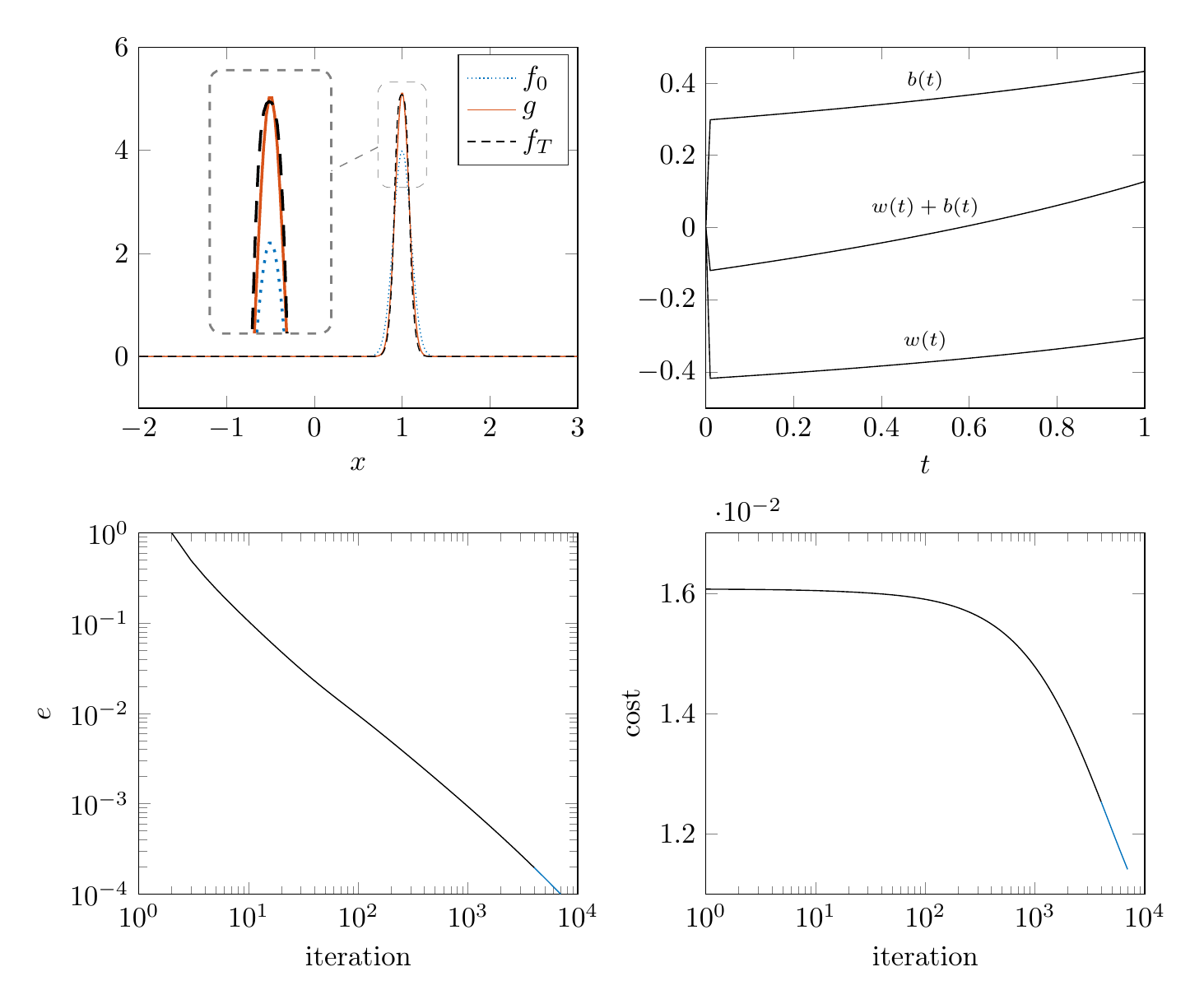}
		\caption{Top left: initial state $f_0$ (blue dotted line), target $g$ (solid red line) and final state $f_T$ (black dashed line) at final time $T=1$. Top right: optimal controls $w(t)$ and $b(t)$, and their sum on the time interval $[0,1]$. Bottom left: relative error of two consecutive iterations of the controls. Bottom right: behavior of the cost functional over the iterations.\label{fig:control2}}
	\end{figure}
	
	In the top left panel of Figure~\ref{fig:control2} we observe that the final state $f_T$ (black dashed line) recovers the target $g$. The optimal controls $w(t)$ and $b(t)$ are shown in the top right panel and chosen when the stopping criterion is met, see the bottom left panel. According to Lemma~\ref{lem:mconserved} and Proposition~\ref{prop:controllab2} we expect to have $w(t)=-b(t)\mu$ and, since $\mu=1$, we notice that indeed $w(t)+b(t)\approx 0$. The relative error between cost functional, see the bottom right panel of Figure~\ref{fig:control2}, is  monotone decreasing.
	
%	\todo[inline]{Problem: $-\int_0^T w(t) \de t = 0.35$, whereas $\alpha = 0.25$ (cf.~Proposition 3). If I start from the optimal prescribed by Proposition 3, that is $w(t)=-0.25$ and $b(t)=-w(t)$, Armijo does not converge in one step, and again $-\int_0^T w(t) \de t = 0.35$. So, I guess that I am not solving correctly the balance law in Proposition 3.}
	
	\paragraph{Test 3.} In the last numerical example we build an artificial test and consider exact controls
	$$
		w_e(t) = e^t-1, \quad b_e(t) = -5t^2+t
	$$
	to evolve the PDE~\eqref{mh-0001} up to time $T=1$ starting from a Beta distribution as initial condition:
	$$
		f_0(x) = \frac{x^{a_1-1} (1-x)^{a_2-1}}{B(a_1,a_2)}
	$$
	where $B$ is the Beta function, and $a_1=2$, $a_2=5$. We obtain a numerical final state that we use as target to initialize the adjoint equation~\eqref{mh-0002}. Finally, the optimality system is solved with the block Gauss--Seidel approach in order to recover the exact controls $w(t)$ and $b(t)$. The physical domain is again $x\in\Omega=[-2,3]$. We consider a fixed time step $\Delta t = 10^{-2}$ and the spatial domain is discretized with $400$ cells. The regularization parameter is considered different for the two controls, precisely we set $\gamma_w = 1$ and $\gamma_b = 10^{-4}$. The tolerance for the stopping criterion is $TOL = 10^{-4}$. The maximum number of iteration of the Armijo--stepsize rule is $10$. For this numerical test we choose the sigmoid activation function, i.e.~$\sigma(x)=\frac{1}{1+\exp(-x)}$. 
	
	\begin{figure}[t!]
		\centering
		\includegraphics[width=\textwidth]{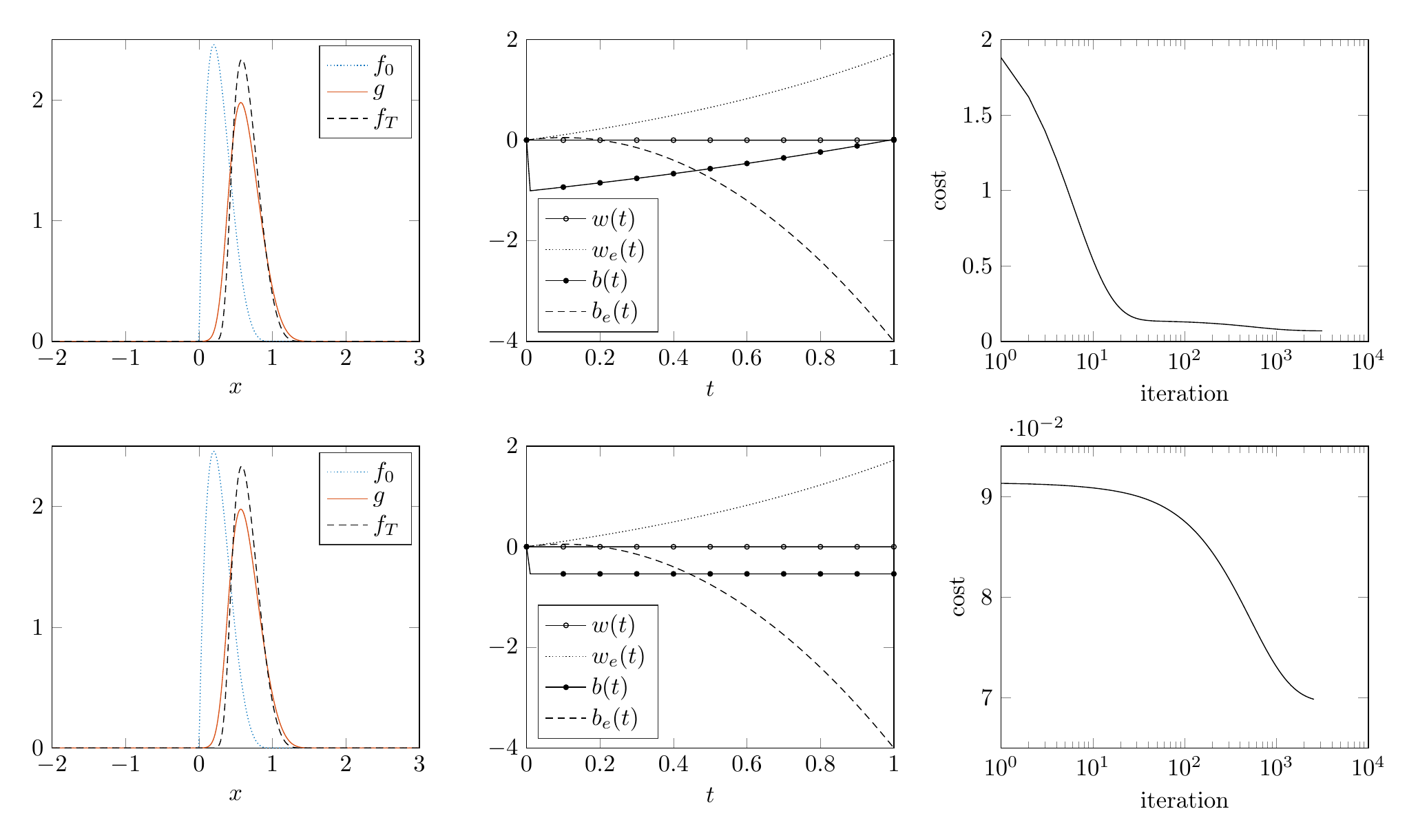}
		\caption{Top row: initial guess for the controls is $w(t)=b(t)=t$, $\forall\,t\in[0,1]$. Bottom row: initial guess for the controls is $w(t)=b(t)=0$, $\forall\,t\in[0,1]$. Left column: initial state $f_0$ (blue dotted line), target $g$ (solid red line) and final state $f_T$ (black dashed line) at time $T=1$. Center column: optimal controls $w(t)$ and $b(t)$, and the exact controls $w_e(t)$ and $b_e(t)$ on the time interval $[0,1]$. Right column: behavior of the cost functional over the iterations.\label{fig:control3}}
	\end{figure}

	In Figure~\ref{fig:control3} we show the results of the numerical experiment obtained with two different initial controls. In particular, the top row panels refer to $w^0(t)=b^0(t)=t$, $\forall\,t\in[0,1]$, whereas the bottom row panels refer to the case $w^0(t)=b^0(t)=0$, $\forall\,t\in[0,1]$. We notice that in both cases the final state $f_T$, black dashed line in the top left panel, reproduces the expected value of the target, but the method is failing in estimating the  variance and the height of the extremal point. This is a consequence of the particular choice of the loss function, as already pointed out in Remark~\ref{rem:loss}.  The optimal controls computed at the end of the optimization procedure are shown in the center column panels and compared with the exact controls $w_e(t)$ and $b_e(t)$. In both cases, the method provides a constant weight, precisely $w(t)\approx 10^{-3}$, whereas $b(t)$ differs, depending on the choice of the initial guess. This result show the possible existence of multiple optimal controls solving the same task of recovering the target $g$.
		
	\section{Conclusion and Future Work} \label{sec:conclusion}
	
	In this work we have proposed and analyzed a mean--field description of residual neural networks. The limit is performed on the number of data, and the well--posedness of the resulting Vlasov--type equation is discussed. We have proved existence and uniqueness of weak solutions, continuous dependence on the initial condition and on the parameters, and the convergence of the solution of the discrete system to the solution of the PDE.
	
	Furthermore, we have tackled the problem of the training of the mean--field neural network using a controllability and an optimal control point of view. We have shown existence of the minimizers and proposed a computational approach based on first--order optimality conditions to numerically optimize the unknown parameters. Finally, we have performed numerical experiments on the derived equations.
	
	We expect that further analysis of the  mathematical formulations of machine learning models at different scales is a useful tool to break the complexity of the methods on discrete level and to provide theoretical foundations, in--depth understanding, analysis and improvements of existing approaches. In particular, the present work opens several research perspectives, as for instance the study of the convergence of the optimal solutions of the discrete training process to the solutions of the mean--field optimal control problem via Gamma--convergence, or the definition of different loss functions at the mean--field level and the computational comparison between the discrete and the mean--field training.
	
	\section*{Acknowledgments}
	M.H. thanks the Deutsche Forschungsgemeinschaft (DFG, German Research Foundation) for the financial support through 20021702/GRK2326,  333849990/IRTG-2379, HE5386/18-1,19-1,22-1,23-1 and under Germany’s Excellence Strategy EXC-2023 Internet of Production 390621612.
	G.V. acknowledges the support of the INdAM--GNCS group and of the PRIN2017 project 2017KKJP4X funded by the Italian MUR (Ministry of University and Research). The work of A.T. was supported by a postdoc fellowship of the German Academic Exchange Service (DAAD) (PKZ 91817986).

	%-- LITERATUR ----------------------------------------------------------%
	%\clearpage
	\bibliography{literaturmean.bib}
	\bibliographystyle{abbrv}

\end{document}